\documentclass[12pt]{article}

\newtheorem{theorem}{Theorem}[section] 
\newtheorem{lemma}[theorem]{Lemma}     
\newtheorem{corollary}[theorem]{Corollary}
\newtheorem{proposition}[theorem]{Proposition}
\newtheorem{definition}{Definition}
\newtheorem{construction}{Construction}

\usepackage{latexsym,amsfonts,amsmath,amssymb,mathrsfs}

\newtheorem{remark}{Remark}

\newcommand{\rd}{\,\mathrm{d}}

\newcommand{\bsb}{\boldsymbol{b}}
\newcommand{\bsx}{\boldsymbol{x}}
\newcommand{\bsy}{\boldsymbol{y}}
\newcommand{\bsz}{\boldsymbol{z}}

\newcommand{\bsell}{\boldsymbol{\ell}}
\newcommand{\bsk}{\boldsymbol{k}}

\newcommand{\bsa}{\boldsymbol{a}}

\newcommand{\bssigma}{\boldsymbol{\sigma}}

\newcommand{\bstheta}{\boldsymbol{\theta}}
\newcommand{\bszero}{\boldsymbol{0}}
\newcommand{\bsone}{\boldsymbol{1}}

\newcommand{\nat}{\mathbb{N}}

\newcommand{\LL}{{\cal L}}

\newcommand{\FF}{\mathbb{F}}
\newcommand{\NN}{\mathbb{N}}

\newcommand{\Mod}{\operatorname{mod}}
\newcommand{\mypmod}[1]{\,(\Mod\,#1)}
\newcommand{\wal}{\mathrm{wal}}
\newcommand{\cP}{\mathcal{P}}
\newcommand{\cS}{\mathcal{S}}
\newenvironment{proof}{\begin{trivlist}
    \item[\hskip\labelsep{\it Proof.}]}{$\hfill\Box$\end{trivlist}}
\newcommand{\satop}[2]{\stackrel{\scriptstyle{#1}}{\scriptstyle{#2}}}

\allowdisplaybreaks

\title{Discrepancy bounds for infinite-dimensional order two digital sequences over $\mathbb{F}_2$}

\author{J. Dick}

\date{}

\begin{document}
\maketitle

\begin{abstract}
We provide explicit constructions of infinite-dimensional digital sequences $\mathcal{S} = (\boldsymbol{x}_0, \boldsymbol{x}_1, \ldots, ) \subset [0,1]^{\mathbb{N}}$, which are constructed over the finite field $\mathbb{F}_2$, whose projection onto the first $s$ coordinates $\boldsymbol{x}_0^{(s)}, \boldsymbol{x}_{1}^{(s)}, \ldots$ for all $s \ge 1$, has $\mathcal{L}_q$ discrepancy bounded by
\begin{equation*}
\mathcal{L}_q(\{\boldsymbol{x}^{(s)}_0, \boldsymbol{x}^{(s)}_1, \ldots, \boldsymbol{x}^{(s)}_{N-1}\} ) \le C_{q,s} \frac{r^{3/2-1/q} }{N} \sqrt{ \sum_{v=1}^r m_v^{s-1} }
\end{equation*}
for all $N = 2^{m_1} + 2^{m_2} + \cdots + 2^{m_r} \ge 2$ and even integers $q$ with $2 \le q < \infty$, where the constant $C_{q,s} > 0$ is independent of $N$. In particular, we have
\begin{equation*}
\mathcal{L}_q(\{\boldsymbol{x}^{(s)}_0, \boldsymbol{x}^{(s)}_1, \ldots, \boldsymbol{x}^{(s)}_{2^m-1}\} ) \le C_{q,s} \frac{m^{(s-1)/2}}{2^m}
\end{equation*}
for all $m, s \ge 1$ and $2 \le q < \infty$. Further we give explicit constructions of finite point sets $\boldsymbol{y}_0, \boldsymbol{y}_1, \ldots, \boldsymbol{y}_{N-1}$ in $[0,1)^\mathbb{N}$ for all $N \ge 2$ such that their projection on the first $s$ coordinates $\boldsymbol{y}_0^{(s)}, \boldsymbol{y}_1^{(s)}, \ldots, \boldsymbol{y}_{N-1}^{(s)}$ in $[0,1)^s$ for all $s \ge 1$ satisfies
\begin{equation*}
\mathcal{L}_q(\{\boldsymbol{y}^{(s)}_0, \boldsymbol{y}^{(s)}_1, \ldots, \boldsymbol{y}^{(s)}_{N-1}\} ) \le C_{q,s} \frac{(\log N)^{(s-1)/2}}{N}
\end{equation*}
for all $2 \le q < \infty$, where $C_{q,s} > 0$ is again independent of $N$. The last two results are best possible by a lower bound of Roth [ K. F. Roth, On irregularities of distribution. Mathematika, {\bf 1} (1954), 73--79.].

The proofs are based on a generalization of the Niederreiter-Rosenbloom-Tsfasman metric, which itself is a generalization of the Hamming metric.
\end{abstract}



\section{Introduction}

The $\mathcal{L}_q$ discrepancy is a measure of the equidistribution properties of a point set $\widehat{\cP}_{N,s} = \{ \bsx^{(s)}_0,\bsx^{(s)}_1, \ldots, \bsx^{(s)}_{N-1}\}$ in the unit cube $[0,1]^s$, see \cite{BC,kuinie,mat}. It is based on the local discrepancy function
\begin{equation*}
\delta(\widehat{\cP}_{N,s}; \bstheta) = \frac{1}{N} \sum_{n=0}^{N-1} 1_{[\bszero, \bstheta)}(\bsx_n) - \prod_{j=1}^s \theta_j,
\end{equation*}
where $\bstheta = (\theta_1,\ldots, \theta_s)$, $[\bszero, \bstheta) = \prod_{j=1}^s [0, \theta_j)$, and $1_{[\bszero, \bstheta)}$ denotes the characteristic function of the interval $[\bszero, \bstheta)$. For a given interval $[\bszero, \bstheta)$, the local discrepancy function measures the difference between the proportion of points which fall into this interval and the volume of the interval. The $\mathcal{L}_q$ discrepancy is then the $\mathcal{L}_q$ norm of the discrepancy function
\begin{equation*}
\mathcal{L}_q(\widehat{\cP}_{N,s}) = \left(\int_{[0,1]^s} |\delta(\widehat{\cP}_{N,s}, \bstheta)|^q \,\mathrm{d} \bstheta \right)^{1/q},
\end{equation*}
with the obvious modifications for $q = \infty$. One of the questions on irregularities of distribution is concerned with the precise order of convergence of the smallest possible values of $\mathcal{L}_q(\widehat{\cP}_{N,s})$ as $N$ goes to infinity. That is, the aim is to study the convergence of
\begin{equation*}
\mathcal{L}_{q,N,s} = \inf_{\satop{\widehat{\cP}_{N,s} \subset [0,1]^s}{|\widehat{\cP}_{N,s}|=N}} \mathcal{L}_q(\widehat{\cP}_{N,s}),
\end{equation*}
as $N$ tends to infinity (for fixed dimension $s$) and the explicit construction of point sets $\widehat{\cP}_{N,s}$ which achieve the optimal rate of convergence of the $\mathcal{L}_q$ discrepancy \cite{BC}. (Such point sets are of use for instance in quasi-Monte Carlo integration \cite{DP10,DT97,niesiam}.)

In the next subsection we describe the results of this paper.

\subsection{The results}

Let $\mathbb{N}$ denote the set of natural numbers and $\mathbb{N}_0$ the set of nonnegative integers.

In the following we write $A(N,m,q,s) \ll_{q,s} B(N,m,q,s)$ if there is a constant $c_{q,s} > 0$ which depends only on $s$ and $q$ (but not on $N$ or $m$) such that $A(N,m,q, s) \le c_{q,s} B(N,m,q, s)$ for all $m$ and $N$, with analogous meanings for $\ll_s, \gg_{q,s}, \gg_s$. We write $\cS = (\bsx_0, \bsx_1,\ldots) \subset [0,1)^{\mathbb{N}}$ for an infinite dimensional sequence and $\cS_s = (\bsx_0^{(s)}, \bsx_1^{(s)}, \ldots) \subset [0,1)^s$ for the projection of $\cS$ onto the first $s$ coordinates. Further let $\cP_N = \{\bsx_0, \bsx_1, \ldots, \bsx_{N-1} \} \subset [0,1)^{\mathbb{N}}$ denote the first $N$ points of $\cS$ and let $\cP_{N,s} = \{\bsx_0^{(s)}, \bsx_1^{(s)}, \ldots, \bsx_{N-1}^{(s)}\} \subset [0,1)^s$ denote the first $N$ points of $\cS_s$. For point sets consisting of $N$ elements which are not obtained as the first $N$ points of a sequence for all $N \in \mathbb{N}$, we write $\widehat{\cP}_N$ if the point set is in $[0,1)^{\mathbb{N}}$ and $\widehat{\cP}_{N,s}$ for point sets in $[0,1)
^s$, where $\widehat{P}_{N,s}$ is the projection of $\widehat{P}_N$ onto the first $s$ coordinates.

We show the following theorem.
\begin{theorem}\label{thm1}
One can explicitly construct an infinite sequence $\cS$ of points in $[0,1)^{\mathbb{N}}$ such that for all $s \ge 1$, the projection of the first $N$ points of $\cS$ onto the first $s$ coordinates $\cP_{N,s}$ satisfies
\begin{equation*}
\mathcal{L}_{q}(\cP_{N,s}) \ll_{q,s}  \frac{r^{3/2-1/q} }{N} \sqrt{\sum_{v=1}^r m_v^{s-1}}
\end{equation*}
for all $N \in \mathbb{N}$, with $N \ge 2$, and with dyadic expansion $N = 2^{m_1} + 2^{m_2} + \cdots + 2^{m_r}$, where $m_1 > m_2 > \cdots > m_r \ge 0$, and all even integers $q$ with $2 \le q < \infty$. In particular, we have
\begin{equation*}
\mathcal{L}_{q}(\cP_{2^m,s}) \ll_{q,s} \frac{m^{(s-1)/2}}{2^m} \quad \mbox{for all } m \ge 1 \mbox{ and } 2 \le q < \infty.
\end{equation*}
\end{theorem}

As a corollary to Theorem~\ref{thm1}, using an idea from \cite{CS02}, we can also obtain explicit constructions of finite point sets in the infinite dimensional unit cube $[0,1]^{\mathbb{N}}$ whose projection onto the first $s$ coordinates achieve the optimal rate of convergence of the $\mathcal{L}_q$ discrepancy.
\begin{corollary}\label{cor1}
For every $N \ge 2$ one can explicitly construct a point set  $\widehat{\cP}_{N}$ of $N$ points in $[0,1)^{\mathbb{N}}$ such that for all $s \ge 1$, the projection of $\widehat{\cP}_N$ onto the first $s$ coordinates $\widehat{\cP}_{N,s}$ satisfies
\begin{equation*}
\mathcal{L}_{q}(\widehat{\cP}_{N,s}) \ll_{q,s}  \frac{(\log N)^{(s-1)/2}}{N} \quad \mbox{for all } 2 \le q < \infty.
\end{equation*}
\end{corollary}

In the next subsection we provide a review of the literature and explain how the above results relate to what is known.

\subsection{Literature review}

The classic lower bound on the $\mathcal{L}_q$ discrepancy is by Roth~\cite{Roth} and ascertains that
\begin{equation*}
\mathcal{L}_{q,N,s} \gg_{s} \frac{(\log N)^{(s-1)/2}}{N} \quad \mbox{for all } N, q, s \ge 2.
\end{equation*}
This result is known to be best possible for $q = 2$ as shown first by Davenport~\cite{dav} for $s=2$ and then by Roth~\cite{roth2,Roth4}. Other constructions of point sets with optimal $\mathcal{L}_2$ discrepancy were found by Chen~\cite{C80, C83}, Frolov~\cite{Frolov}, Dobrovol'ski\v{i}~\cite{Do84}, Skriganov~\cite{Skr89, Skr94}, Hickernell and Yue~\cite{HY00}, and Dick and Pillichshammer~\cite{DP05b}. For more details on the history of the subject see the monograph \cite{BC}. All the constructions mentioned so far involve some random elements, except for the special case of $s=2$ studied by Davenport. Further examples of two-dimensional point sets with best possible order of $\LL_2$ discrepancy can be found in
\cite{FauPi09a,FauPi09,FauPiPriSch09,KriPi2006,lp,pro1988a}. Thus the constructions for $s \ge 3$ are not explicit. First explicit constructions of finite point sets in fixed dimension matching the lower bound were provided by the works of Chen and Skriganov \cite{CS02} for $q=2$ and Skriganov~\cite{Skr} for $2 \le q < \infty$. See also Chen and Skriganov~\cite{CS08} where the arguments of \cite{CS02} were simplified and the constant was improved. The papers \cite{CS02} and \cite{Skr} completely solved the open problem of finding explicit constructions of finite point sets of fixed dimension with optimal $\mathcal{L}_2$ and optimal $\mathcal{L}_q$ discrepancy. On the other hand, the $\mathcal{L}_\infty$ discrepancy, called star discrepancy, is much harder to analyze, the exact order of convergence is not known \cite{BL,BLV}.

We briefly describe what is known about the discrepancy of sequences. A lower bound for infinite sequences of points was shown by Pro{\u\i}nov~\cite{pro85}, which states that for all infinite sequences $\cS_s$ in the unit cube $[0,1)^s$ one has
\begin{equation}\label{ineq_pro}
\mathcal{L}_2(\cP_{N,s}) \gg_s \frac{(\log N)^{s/2}}{N},
\end{equation}
for infinitely many values of $N$. This implies that one cannot construct an infinite sequence of points such that its first $N$ points match Roth's lower bound for all values of $N$. An explicit construction of an infinite sequence of points $\cS_s$ in $[0,1]^s$ which satisfies
\begin{equation*}
\mathcal{L}_2(\cP_{N,s}) \ll_s \frac{(\log N)^{s/2}}{N} \quad \mbox{for all } N \ge 2,
\end{equation*}
was provided in \cite{DP12}. Note that those results only apply to the $\mathcal{L}_2$ discrepancy. Further, the sequences from \cite{DP12} match Roth's lower bound for infinitely many values of $N$, more precisely, for $N = 2^m$ one obtains
\begin{equation*}
\mathcal{L}_2(\cP_{2^m,s}) \ll_s \frac{m^{(s-1)/2}}{2^m} \quad \mbox{for all } m \ge 1.
\end{equation*}
One-dimensional infinite sequences whose $\LL_2$ discrepancy satisfies a bound of order $\sqrt{\log N}/N$ for every $N \ge 2$ were given in, e.g. \cite{chafa,g96,lp,pro85,pg}. These constructions are mainly based on the symmetrization of sequences (also called reflection principle).

The explicit construction of sequences studied in \cite{DP12} are the same as in this paper. In \cite{DP12} the authors studied the $\LL_2$ discrepancy, whereas here we consider the $\mathcal{L}_q$ discrepancy for $2 \le q < \infty$. 

Using the estimations $r \le \log N$ and $m_h \le \log N$, the first result of Theorem~\ref{thm1} implies that
\begin{equation}\label{Lq_seq}
\mathcal{L}_q(\cP_{N,s}) \ll_{q,s} \frac{(\log N)^{s/2 + 3/2-1/q}}{N},
\end{equation}
for all $N \ge 2$ and all even integers $q$ with $2 \le q < \infty$. Thus, at least for $q=2$, this result is not best possible. It seems reasonable to suggest that the exponent of the $\log N$ factor above can be replaced by $s/2$, which would be best possible by the lower bound of Pro{\u\i}nov~\cite{pro85}.

On the other hand, for finite point sets, the second part of Theorem~\ref{thm1} and Corollary~\ref{cor1} match the lower bound by Roth~\cite{Roth} and are therefore optimal. The construction of the sequences and point sets presented in this paper uses the finite field $\mathbb{F}_2$, which is different from the construction in \cite{CS02,Skr}, where the points were constructed using the finite field $\mathbb{F}_p$ of prime order $p$ with $p \ge q s^2$. By removing the restriction $p \ge q s^2$ we can now use the projection of infinite dimensional point sets to obtain point sets with optimal $\mathcal{L}_q$ discrepancy, which is not possible using the construction from \cite{CS02, Skr}. Further, the bound on $\mathcal{L}_q(\cP_{N,s})$ holds for all $2 \le q < \infty$, i.e., as opposed to \cite{Skr}, one does not have to change the point set as $q$ increases. Hence the explicit construction in this paper is the first construction which achieves the optimal rate of convergence of the $\mathcal{L}_q$ discrepancy for all $2 \le q < 
\infty$. In fact, we conjecture that the explicit constructions of the point sets and sequences in this paper also achieve the optimal rate of convergence of the $\LL_\infty$ discrepancy.


In the next subsection we describe the explicit construction of sequences and point sets which satisfy Theorem~\ref{thm1} and Corollary~\ref{cor1}.

\subsection{Explicit construction of sequences}\label{sec_const}

The construction is done in two steps. In the first step, we use explicit constructions of so-called digital $(t,m,s)$-nets and digital $(t,s)$-sequences \cite{DP10,nie87,nie88,niesiam,NX96,sob67} over the finite field $\mathbb{F}_2$. We introduce the relevant background as well as a special case of a suitable explicit construction in the following.

We first introduce some notation. We call $x\in [0,1)$ a dyadic rational if it can be written in a finite dyadic expansion. By $\oplus$ we denote the digit-wise addition modulo $2$, i.e., for $x, y \in \mathbb{R}_+ = \{z \in \mathbb{R}: z \ge 0\}$  and dyadic expansions $x = \sum_{i=w}^{\infty} \frac{x_i}{2^i}$ and $y = \sum_{i=w}^{\infty} \frac{y_i}{2^i}$ for some $w \in \mathbb{Z}$, we have $$ x \oplus y := \sum_{i=w}^{\infty } \frac{z_i}{2^i}, \; \; \;{\rm where} \; \; \; z_i := x_i + y_i \mypmod{2},$$ where for dyadic rationals we always use the finite expansion. For vectors $\boldsymbol{x}, \boldsymbol{y} \in \mathbb{R}_+^s$ we use the notation $\bsx \oplus \bsy$ to denote the component-wise addition $\oplus$. (Since we consider addition modulo $2$, the dyadic subtraction $\ominus$ is the same as the dyadic addition $\oplus$.) Note that, for instance, for $x = 2^{-1} + 2^{-3} + 2^{-5} + \cdots$ and $y = 2^{-2} + 2^{-4} + 2^{-6} + \cdots$ we obtain $x \oplus y = 2^{-1} + 2^{-2} + 2^{-3} + \cdots$, which 
is given by its infinite expansion although it is a dyadic rational. Hence $x \oplus y$ is not always defined via its finite expansion, even if we always use the finite expansion of $x$ and $y$. This problem could be avoided by using the dyadic group $(\mathbb{F}_2)^{\mathbb{N}}$ as in \cite[Section~2]{Fine} instead of $\mathbb{R}_+$. However, this situation does not occur in this paper since we only use $\oplus$ for (vectors of) dyadic rationals (in fact, usually nonnegative integers) for which we always use the finite expansion in our proofs, so it is sufficient to use $\mathbb{R}_+$ instead of $(\mathbb{F}_2)^{\mathbb{N}}$.

\subsubsection*{The digital construction scheme}
We now describe the digital construction scheme for point sets in the unit cube. In the following we identify $0, 1 \in \FF_2$ with the integers $0, 1$. Let $C_j = (c_{j,k,\ell})_{\satop{1 \le k \le 2m}{1 \le \ell \le m}} \in \mathbb{F}_2^{2m \times m}$ for $j \in \mathbb{N}$ be $2m \times m$ matrices over $\mathbb{F}_2$.  Let $n = n_0 + n_1 2 + \cdots + n_{m-1} 2^{m-1} \in \{0, 1, \ldots, 2^m-1\}$ be the dyadic expansion of $n$. Set $\vec{n} = (n_0, n_1, \ldots, n_{m-1})^\top \in \mathbb{F}_2^{m}$. Then define
\begin{equation*}
\vec{x}_{j,n} = C_j \vec{n},
\end{equation*}
that is, $\vec{x}_{j,n} = (x_{j,n,1}, x_{j,n,2},\ldots, x_{j,n,2m} )^\top$ with $x_{j,n,k} = \sum_{\ell=1}^m n_{\ell-1} c_{j,k,\ell} \in \mathbb{F}_2$ and define
\begin{equation*}
x_{j,n} = x_{j,n,1} 2^{-1} + x_{j,n,2} 2^{-2} + \cdots + x_{j,n,2m} 2^{-2m}.
\end{equation*}
Then the $n$th point $\boldsymbol{x}_n$ of the point set is given by $\boldsymbol{x}_n = (x_{1,n}, x_{2,n}, \ldots ) \in [0,1)^{\mathbb{N}}$. The point set $\widehat{\cP}_{2^m} = \{\bsx_0, \bsx_1,\ldots, \bsx_{2^m-1}\}$ is a digital net.

With some minor modifications we can also set $m=\infty$. In this case the generating matrices are of the form $C_j = (c_{j,k,\ell})_{k, \ell \in \mathbb{N}}$ and we obtain an infinite sequence $\cS$, which we call a digital sequence (with generating matrices $(C_j)_{j \in \mathbb{N}}$). In this case we have $x_{j,n,k} = \sum_{\ell=1}^\infty n_{\ell-1} c_{j,k,\ell} \in \mathbb{F}_2$, which is actually a finite sum since for any $n \in \mathbb{N}_0$ only finitely many digits are nonzero. Further, we consider only matrices $C_j = (c_{j,k, \ell})_{k, \ell \in \mathbb{N}}$ for which $c_{j,k,\ell} = 0$ for all $k > 2 \ell$.  (We point out that we actually only need $c_{j,k,\ell} = 0$ for all $k$ large enough for our purposes here, but to simplify the notation we use only constructions for which $c_{j,k,\ell}=0$ for $k > 2 \ell$.)

For a matrix $C_j = (c_{j,k,\ell}) \in \mathbb{F}_2^{\mathbb{N} \times \mathbb{N}}$ we denote by $C_j^{u \times v} = (c_{j,k,\ell})_{1 \le k \le u, 1 \le \ell \le v}$ the left-upper $u \times v$ submatrix of $C_j$.

For the proof we also use the concept of a digitally shifted digital net. Let $\bssigma = (\sigma_1, \sigma_2, \ldots) \in [0,1]^{\mathbb{N}}$ with dyadic expansion $\sigma_j = \sigma_{j,1} 2^{-1} + \sigma_{j,2} 2^{-2} + \cdots$. Then the digitally shifted digital net $\widehat{\cP}_{2^m}(\bssigma)$ consists of the points $\bsx_{n} \oplus \bssigma$ for $0 \le n < 2^m$. (Below we only use shift vectors whose components are dyadic rationals.)  We now show that certain subsets of a digital sequence are digitally shifted digital nets.

\begin{lemma}\label{lem_net_seq}
Let $\bsx_0, \bsx_1, \ldots$ be the points of a digital sequence with generating matrices $C_j = (c_{j,k,\ell})_{k, \ell \in \mathbb{N}}$ for which $c_{j,k,\ell} = 0$ for all $k > 2 \ell$. Let $m \ge 0$. Then for any $\beta \ge 0$ the point set
\begin{equation*}
\bsx_{\beta 2^m}, \bsx_{\beta 2^m+1}, \ldots, \bsx_{\beta 2^m + 2^m-1}
\end{equation*}
is a digitally shifted digital net with generating matrices $C_{j}^{2m \times m}$, $j \in \mathbb{N}$, and which is shifted by a digital shift vector whose coordinates are dyadic rationals.
\end{lemma}

\begin{proof}
For $n \in \{\beta 2^m, \beta 2^m+1, \ldots, \beta 2^m +2^m-1\}$ we write $n = a + \beta 2^m$ with $0 \le a < 2^m$. Then $\vec{n} = (\vec{a}^\top, \vec{0}_\infty^\top)^\top + (\vec{0}_{m}^\top, \vec{\beta}^\top)^\top$, where $\vec{a} = (n_0, n_1,\ldots, n_{m-1})^\top$, $\vec{\beta} = (n_{m}, n_{m+1},\ldots)^\top$ and $\vec{0}_z$ is the zero-vector of length $z$. We write
$$
C_{j} = \left( \begin{array}{ccc}
           & \vline &  \\
  C_{j}^{2 m \times m} & \vline & D_{j}^{2 m \times \NN} \\
           & \vline &   \\ \hline
    & \vline & \\
 0^{\NN \times m}  & \vline &  F_{j}^{\NN \times \NN} \\
   &   \vline       &
\end{array} \right) \in \FF_2^{\NN \times \NN},
$$
where $0^{\NN \times m}$ denotes the $\mathbb{N} \times m$ matrix whose entries are all $0 \in \mathbb{F}_2$. With this notation we have
$$C_j \vec{n}=\left( \begin{array}{c} C_{j}^{2 m \times m}  \vec{a} \\ 0 \\ 0 \\ \vdots \end{array} \right)  + \left( \begin{array}{c}
             \\
  D_{j}^{2 m \times \NN}  \\
          \\ \hline  \\
 F_{j}^{\NN \times \NN}  \\
   \end{array} \right) \vec{\beta}.$$ For the point set under consideration, the vector
\begin{equation*}
\vec{\sigma}_{\beta,j}:=\left( \begin{array}{c}
             \\
  D_{j,2 m \times \NN}  \\
          \\ \hline  \\
 F_{j,\NN \times \NN}  \\
   \end{array} \right) \vec{\beta}
\end{equation*}
is fixed. Let $\vec{\sigma}_{\beta,j} = (\sigma_{\beta,j,1}, \sigma_{\beta,j,2},\ldots)^\top$. By the assumption $c_{j,k, \ell} =0$ for all $k > 2\ell$ it also follows that $\sigma_{\beta,j,b} = 0$ for all $b$ large enough. Further, as $n$ runs through all elements in the set $\{\beta 2^m, \beta 2^m+1,\ldots, (\beta + 1) 2^m-1\}$, the vector $\vec{a}$ runs through all elements in the set $\mathbb{F}_2^m$. Thus the point set $\{\bsx_{\beta 2^m}, \bsx_{\beta 2^m+1},\ldots, \bsx_{\beta 2^m+2^m-1}\}$ is a digitally shifted digital net with generating matrices $C_j^{2m \times m}$, $j \in \mathbb{N}$ and digital shift vector $\bssigma_\beta = (\sigma_{\beta,j})_{j \in \mathbb{N}}$ where $\sigma_{\beta,j} = \sigma_{\beta,j,1} 2^{-1} + \sigma_{\beta,j,2} 2^{-2} + \cdots$ are dyadic rationals.
\end{proof}

\subsubsection*{The NRT weight function}
The properties of the digital sequence $\cS$ depend entirely on the properties of the generating matrices $(C_j)_{j \in \mathbb{N}}$. We now introduce a weight function which serves as a criterion for selecting good generating matrices.  Assume that the integer $k > 0$ has dyadic expansion $k = \kappa_0 + \kappa_1 2 + \cdots + \kappa_{a-2} 2^{a-2} + 2^{a-1}$ with $\kappa_i \in \{0,1\}$.  We define the NRT weight function $\mu_1$ (Niederreiter~\cite{nie86} and Rosenbloom-Tsfasman~\cite{RT}) weight) for nonnegative integers $k$ by
\begin{equation}\label{def_mu}
\mu_1(k) = \left\{\begin{array}{rl} a = 1 + \lfloor \log_2 k \rfloor & \mbox{if } k > 0, \\ 0 & \mbox{if } k = 0, \end{array} \right.
\end{equation}
where $\lfloor x \rfloor$ is the largest integer smaller or equal to $x$. For vectors $\bsk = (k_1,\ldots, k_s) \in \mathbb{N}_0^s$ we define the NRT weight by
\begin{equation*}
\mu_1(\bsk) = \mu_1(k_1) + \mu_1(k_2) + \cdots + \mu_1(k_s).
\end{equation*}

We now explain how the NRT weight is used to obtain a criterion for choosing good generating matrices. For $m \ge 1$ let $C_j^{2m \times m} \in \mathbb{F}_2^{2m \times m}$ denote the left-upper $2m \times m$ sub-matrix of $C_j \in \mathbb{F}_2^{\mathbb{N} \times \mathbb{N}}$. Further we set $\vec{k} = (\kappa_0, \kappa_1, \ldots, \kappa_{2m-1})^\top \in \mathbb{F}_2^{2m}$, where for $a < 2m$ we set $\kappa_{i} = 0$ for $a-1 < i \le 2m-1$. We define
\begin{align*}
\mathcal{D}_{m,s} = & \mathcal{D}(C_1^{2m \times m},\ldots, C_s^{2m \times m}) \\ = & \{\bsk = (k_1,\ldots, k_s) \in \mathbb{N}_0^s: (C_1^{2m \times m})^\top \vec{k}_1 + \cdots + (C_s^{2m \times m})^\top \vec{k}_s = \vec{0}_m \in \mathbb{F}_2^m\},
\end{align*}
where $\vec{0}_m$ denotes the column zero vector in $\FF_2^m$. Further we set $\mathcal{D}^\ast_{m,s} = \mathcal{D}_{m,s} \setminus \{\bszero\}$, where $\bszero$ denotes the zero-vector in $\mathbb{N}_0^s$. (The set $\mathcal{D}_{m,s}$ is related to the dual space of the row space of $((C_1^{2m \times m})^\top, \ldots, (C_s^{2m \times m})^\top)$.)

We define the minimal weight of $\mathcal{D}^\ast_{m,s}$ as
\begin{equation*}
\rho_{1,m,s} = \rho_{1,m,s}(\mathcal{D}^\ast_{m,s}) = \min_{\bsk \in \mathcal{D}_{m,s}^\ast} \mu_1(\bsk).
\end{equation*}
It can be shown that a large weight $\rho_{1,m,s}(\mathcal{D}^\ast_{m,s})$ for all $m \ge 1$ yields good distribution properties of the corresponding digital sequence. Therefore the goal is to construct generating matrices $(C_j)_{j \in \mathbb{N}}$ of digital sequences for which the minimal weight is in some sense large. Since this is only an intermediate step in our construction, we will not go into the details of relating the NRT weight to the distribution properties of the sequence, the interested reader may, for instance, consult \cite{nie87,np} for details.

\subsubsection*{Construction of generating matrices with large $\rho_{1,m,s}$}
We return to the digital construction scheme. We now introduce a construction of generating matrices with large minimal weight $\rho_{1,m,s}$. This is the first step in our construction of obtaining digital sequences which satisfy the bound in Theorem~\ref{thm1}.

Explicit constructions of suitable generating matrices $C_j \in \mathbb{F}_2^{\mathbb{N} \times \mathbb{N} }$ were obtained by Sobol'~\cite{sob67}, Niederreiter~\cite{nie88}, Tezuka~\cite{Tez}, Niederreiter-Xing~\cite{NX96} and others (see also \cite[Chapter~8]{DP10}). To make the construction fully explicit, we briefly describe a special case of generalized Niederreiter sequences introduced by Tezuka~\cite[Eq. (3)]{Tez}. The basic idea of this construction is based on Sobol's and Niederreiter's construction of the generating matrices. The construction is based on irreducible polynomials over the finite field $\mathbb{F}_2$. Let $p_1=x$ and $p_j \in \mathbb{F}_2[x]$, for $j \ge 2$, be the $(j-1)$st irreducible polynomial in a list of irreducible polynomials over $\mathbb{F}_2$ that is sorted in increasing order according to their degree $e_j = \deg(p_j)$, that is, $e_1 \le e_2 \le \cdots $ (the ordering of polynomials with the same degree is irrelevant; further, one could also use primitive polynomials 
instead of
irreducible polynomials).

Let $C_j = (c_{j,k,\ell})_{k,\ell \in \mathbb{N} }$ with $c_{j,k,\ell} \in \mathbb{F}_2$. We describe now how to obtain the element $c_{j,k,\ell}$ for $j, k, \ell \ge 1$. To do so, fix natural numbers $j$ and $k$. Take $i-1$ and $z$ to be respectively the main term and remainder when we divide $k-1$ by $e_j$, so that $k-1 = (i-1) e_j + z$ with $0 \le z < e_j$. Now consider the Laurent series expansion
\begin{equation*}
\frac{x^{e_j-z-1}}{p_j(x)^i} = \sum_{\ell =1}^\infty a_\ell(i,j,z) x^{-\ell} \in \mathbb{F}_2((x^{-1})).
\end{equation*}
Then for all $\ell \ge 1$ we set
\begin{equation*}
c_{j,k,\ell} = a_\ell(i,j,z).
\end{equation*}
Note that in this construction we have $c_{j,k,\ell} = 0$ for all $k > \ell$.

The weight function and constructions we introduced so far have been well studied. In the following we introduce a new weight function and construction of generating matrices which can be viewed as an extension of the constructions above. It has first been studied in \cite{D07,D08}.

\subsubsection*{A new weight function}
As mentioned above, we have not found the NRT weight to be sufficient to obtain explicit constructions of point sets and sequences satisfying the bound in Theorem~\ref{thm1}. In fact, \cite{BC} and \cite{Skr} use the NRT weight and additionally a Hamming weight to obtain their constructions. Here we use a generalization of the NRT weight (but we do not use the Hamming weight). We introduce this weight function in the following. Let $k = 2^{a_1-1} +  2^{a_2-1} + \cdots + 2^{a_\nu-1} \in \mathbb{N}$, where $a_1 > a_2 > \cdots > a_\nu > 0$. Then we define the weight function
\begin{equation}\label{def_mu2}
\mu_2(k) = \left\{\begin{array}{rl} a_1 + a_2 & \mbox{if } \nu \ge 2, \\ a_1 & \mbox{if } \nu = 1, \\ 0 & \mbox{if } k = 0. \end{array}  \right.
\end{equation}
For vectors $\bsk = (k_1,\ldots, k_s) \in \mathbb{N}_0^s$ we set
\begin{equation}\label{def_mu2_vec}
\mu_2(\bsk) = \mu_2(k_1) + \cdots + \mu_2(k_s).
\end{equation}
We can also define the minimal weight by
\begin{equation*}
\rho_{2,m,s} = \rho_{2,m,s}(\mathcal{D}^\ast_{m,s}) = \min_{\bsk \in \mathcal{D}^\ast_{m,s}} \mu_2(\bsk).
\end{equation*}
The main idea in this paper is to use $\rho_{2,m,s}$ as the criterion to choose generating matrices $(C_j)_{j \in \mathbb{N}}$ and to use it to prove Theorem~\ref{thm1}.

In the following we introduce the second part of our construction of digital sequences.

\subsubsection*{Construction of generating matrices with large $\rho_{2,m,s}$}
We first describe a method to obtain generating matrices $(C_j)_{j \in \mathbb{N}}$ for which $\rho_{2,m,s}$ is large. The following definition was used in \cite{D08} to obtain explicit construction of suitable sequences.

\begin{definition}\rm
The digit interlacing composition is defined by
\begin{eqnarray*}
\mathscr{D}: [0,1)^{2} & \to & [0,1) \\
(x_1, x_{2}) &\mapsto & \sum_{d=1}^\infty \sum_{r=1}^2
\xi_{r,d} 2^{-r - 2 (d-1)},
\end{eqnarray*}
where $x_r = \xi_{r,1} 2^{-1} + \xi_{r,2} 2^{-2} + \cdots$ for $1
\le r \le 2$. We also define this function for vectors by setting
\begin{eqnarray*}
\mathscr{D}: [0,1)^{\mathbb{N}} & \to & [0,1)^\mathbb{N} \\
(x_1, x_2, \ldots) &\mapsto & (\mathscr{D}(x_1, x_2),  \mathscr{D}(x_{3},x_{4}), \ldots),
\end{eqnarray*}
for point sets $\cP_{N} = \{\bsx_0,\bsx_1, \ldots, \bsx_{N-1}\} \subseteq [0,1)^{\mathbb{N} }$ by setting
\begin{equation*}
\mathscr{D}(\cP_{N}) = \{\mathscr{D}(\bsx_0), \mathscr{D}(\bsx_1), \ldots, \mathscr{D}(\bsx_{N-1})\}\subseteq[0,1)^{\mathbb{N}}
\end{equation*}
and sequences $\cS = (\bsx_0, \bsx_1, \ldots)$ by setting
\begin{equation*}
\mathscr{D}(\cS) = (\mathscr{D}(\bsx_0), \mathscr{D}(\bsx_1), \ldots).
\end{equation*}
\end{definition}

We comment here that the interlacing can also be applied to the generating matrices $C_1, C_2, \ldots$ of digital nets or digital sequences directly as described in \cite[Section~4.4]{D08}. This is done in the following way: Let $C_1, C_2, \ldots$ be generating matrices of a digital net or digital sequence and let $\vec{c}_{j,k}$ denote the $k$th row of $C_j$. We define matrices $D_1, D_2, \ldots$, where the $k$th row of $D_j$ is given by $\vec{d}_{j,k}$, in the following way: For all $j \ge 1$, $u \ge 0$ and $1 \le v \le 2$ let
\begin{equation*}
\vec{d}_{j,2 u  + v} = \vec{c}_{2 (j-1)  + v, u+1}
\end{equation*}
It is easy to show that if $C_1, C_2, \ldots$ are the generating matrices of a digital net $\cP_{N}$ or digital sequence $\cS$ respectively, then the matrices $D_1, D_2, \ldots$ defined above, are the generating matrices of $\mathscr{D}(\cP_{N})$ or $\mathscr{D}(\cS)$ respectively. In particular, $\mathscr{D}(\cP_{N})$ is a digital net and $\mathscr{D}(\cS)$ is a digital sequence.

In the proof of Theorem~\ref{thm1} below we show that the following explicit construction satisfies the $\mathcal{L}_q$ discrepancy bounds:
\begin{construction}\label{construction}
Let $C_j \in \mathbb{F}_2^{\mathbb{N} \times \mathbb{N} }$ be defined as above (based on a special case of Tezuka's construction \cite{Tez} of generalized Niederreiter sequences as described above) and let $\cS$ in $[0,1)^{\mathbb{N} }$ denote the digital sequence obtained from these generating matrices. Then the sequence $\mathscr{D}(\cS) \subset [0,1)^{\mathbb{N}}$ provides an example of an explicit construction of a sequence satisfying Theorem~\ref{thm1}.
\end{construction}

Let $\cS$ be a digital sequence as defined in Construction~\ref{construction}. Since the generating matrices $C_j = (c_{j,k,\ell})_{k, \ell \in \mathbb{N}}$ of $\cS$ satisfy $c_{j, k, \ell} = 0$ for $k > \ell$, the generating matrices $D_j = (d_{j,k,\ell})_{j,k,\ell \in \mathbb{N}}$ for the sequence $\mathscr{D}(\cS)$ satisfy $d_{j,k,\ell} = 0$ for $k > 2\ell$.

Finally we describe how to obtain, for each $N \in \mathbb{N}$, a finite point set $\widehat{\cP}_{N} \subset [0,1]^{\mathbb{N}}$, whose projection onto the first $s$ coordinates achieves the optimal order of convergence of the $\mathcal{L}_q$ discrepancy. To do so, we use a propagation rule introduced in \cite{CS02}. In Section~\ref{cor_N} we show that the subset
\begin{equation}\label{P_tilde}
\widetilde{\cP}_{N,s}:= \mathscr{D}(\cP_{2^m,2s}) \cap
\left(\left[0,\frac{N}{2^m}\right) \times [0,1)^{s-1}\right)
\end{equation}
contains exactly $N$ points. Then we define the point set
\begin{equation}\label{Npoints}
\widehat{\cP}_{N,s}:=\left\{\left(\frac{2^m}{N} x_1,x_2,\ldots,x_s\right)\, :
\, (x_1,x_2,\ldots,x_s) \in \widetilde{\cP}_{N,s}\right\}.
\end{equation}
Further we show in Section~\ref{cor_N} that the point set $\widehat{\cP}_{N,s}$ satisfies the bound in Corollary~\ref{cor1}.

\subsection{The essential property}

The construction in the previous subsection is a special case of a more general construction principle for infinite-dimensional sequences which satisfy Theorem~\ref{thm1}. We describe this in the following.

\begin{definition}\rm\label{def_net}
Let $m \ge 1$ and $0 \le t \le 2 m$  be  natural
numbers. Let $\FF_2$ be the finite field of order $2$ and
let $C_1,\ldots, C_s \in \FF_2^{2 m \times m}$ with $C_j = (c_{j,1},
\ldots, c_{j, 2 m})^\top$. If for all $1 \le i_{j,\nu_j} < \cdots <
i_{j,1} \le 2 m$ with $$\sum_{j = 1}^s \sum_{l=1}^{\min(\nu_j,2)} i_{j,l}  \le
2 m - t$$ the vectors
$$c_{1,i_{1,\nu_1}}, \ldots, c_{1,i_{1,1}}, \ldots,
c_{s,i_{s,\nu_s}}, \ldots, c_{s,i_{s,1}}$$ are linearly independent
over $\FF_2$, then the digital net with generating matrices
$C_1,\ldots, C_s$ is called an order $2$ digital $(t,m,s)$-net over $\FF_2$.
\end{definition}

\begin{definition}\rm\label{def_seq}
Let $t \ge 0$ be an integer. Let $C_1,\ldots, C_s \in \FF_2^{\mathbb{N} \times \mathbb{N}}$ and let $C_{j}^{2 m \times m}$ denote the left upper $2 m \times m$ submatrix of  $C_j$. If for all $m > t/2$ the matrices $C_{1}^{2 m \times m},\ldots, C_{s}^{2 m \times m}$ generate an order $2$ digital $(t, m,s)$-net over $\FF_2$, then the digital sequence with generating matrices $C_1,\ldots, C_s$ is called an {\it order $2$ digital $(t,s)$-sequence over $\FF_2$}.
\end{definition}

From \cite[Lemma~4]{Tez} we obtain that generalized Niederreiter sequences are digital $(t',s)$-sequences with
\begin{equation*}
t' = \sum_{j=1}^s (e_j-1).
\end{equation*}
A special case of \cite[Theorem~4.12]{D08} is the following result (set $d=2$ in \cite[Theorem~4.12]{D08}).
\begin{theorem}
Let $\cS$ be a digital sequence such that $\cS_s$ is a digital $(t',s)$-sequence. Then the projection of the sequence $\mathscr{D}(\cS) \subset [0,1)^{\mathbb{N}}$ onto the first $s \ge 1$ coordinates is an order $2$ digital $(t,s)$-sequence over $\mathbb{F}_2$ with
\begin{equation*}
t = s + 2 t'.
\end{equation*}
\end{theorem}

The main property of order $2$ digital $(t,m,s)$-nets with generating matrices $C_1, \ldots, C_s \in \mathbb{F}_2^{2 m \times m}$ is that the minimum weight of $\mathcal{D}(C_1,\ldots, C_s)$ satisfies
\begin{equation}\label{2min_weight}
\rho_{2,m,s}(\mathcal{D}^\ast(C_1,\ldots, C_s)) > 2m-t.
\end{equation}
This property follows directly from the linear independence property of the rows of the generating matrices. Consider now the sequence $\mathscr{D}(\cS)$. \cite[Proposition~1]{D10} implies that the first $2^m$ points of the projection of $\mathscr{D}(\cS)$ onto the first $s$ coordinates is also a digital $(t,m,s)$-net. Thus the linear independence properties of certain sets of rows of the generating matrices implies that we also have
\begin{equation}\label{min_weight}
\rho_{1,m,s}(\mathcal{D}^\ast(C_1,\ldots, C_s)) > m- t.
\end{equation}

\subsubsection*{Some background on higher order nets}

Definitions~\ref{def_net} and \ref{def_seq} are derived from numerical integration of smooth functions studied in \cite{D08}. We give only a very rough description of these results in the following, since we do not rely on them for our purposes here. Let $\bsx_0, \bsx_1, \ldots, \bsx_{2^m-1} \in [0,1]^s$ be an order $2$ digital $(t,m,s)$-net over $\mathbb{F}_2$. Let $f:[0,1]^s \to \mathbb{R}$ be a function whose partial mixed derivatives up to order $2$ in each variable are square integrable, that is,
\begin{equation*}
\int_{[0,1]^s} \left|\frac{\partial^{\boldsymbol{\tau} } f}{\partial \bsx^{\boldsymbol{\tau} }}(\bsx) \right|^2 \rd \bsx < \infty,
\end{equation*}
where for $\boldsymbol{\tau} = (\tau_1,\tau_2,\ldots, \tau_s) \in \{0, 1, 2\}^s$, the expression $\frac{\partial^{\boldsymbol{\tau} } f}{\partial \bsx^{\boldsymbol{\tau} }}(\bsx)$ denotes the partial mixed derivatives of order $\tau_j$ in coordinate $j$. Then
\begin{equation*}
\left|\int_{[0,1]^s} f(\bsx) \rd \bsx - \frac{1}{2^m} \sum_{n=0}^{2^m-1} f(\bsx_n) \right| \ll_{f,s,t} \frac{m^{2 s}}{2^{2 m}}.
\end{equation*}
See \cite{D08} for details.

\section{A bound on the $\mathcal{L}_q$ discrepancy of higher order digital sequences}\label{cor_N}

In this section we state a bound on the $\mathcal{L}_q$ discrepancy of the higher order digital sequences introduced in Section~\ref{sec_const}. Construction~\ref{construction} in Section~\ref{sec_const} is infinite dimensional, however in this section we only deal with the projection of those infinite dimensional sequences onto the first $s$ coordinates. To simplify the notation we write $\bsx_n$ instead of $\bsx_n^{(s)}$ in the remainder of the paper. The next theorem implies the first part of Theorem~\ref{thm1}.

\begin{theorem}\label{thm2}
For all even integers $q$ with $2 \le q < \infty$, the $\mathcal{L}_q$ discrepancy of the first $N \ge 2$ points of an order $2$ digital $(t,s)$-sequence $\cS_{s}$ in $[0,1)^s$ over $\mathbb{F}_2$ is bounded by
\begin{align*}
\mathcal{L}_q(\cP_{N,s})  \ll_{s, q} & \frac{r^{3/2-1/q} }{N} \sqrt{\sum_{v=1}^r m_v^{s-1} },
\end{align*}
where $N = 2^{m_1} + 2^{m_2} + \cdots + 2^{m_r}$ with $m_1 > m_2 > \cdots >  m_r \ge 0$.
\end{theorem}

The proof of this result is shown in Section~\ref{sec_appendix}. Choosing $r=1$ in Theorem~\ref{thm2} yields the following corollary. This result implies the second part of Theorem~\ref{thm1}.
\begin{corollary}
Let $\cP_{2^m,s}$ be an order $2$ digital net. Then
\begin{equation*}
\mathcal{L}_q(\cP_{2^m,s}) \ll_{s,q} \frac{m^{(s-1)/2}}{2^m} \quad \mbox{for all } 2 \le q < \infty.
\end{equation*}
\end{corollary}

The next corollary shows that the optimal convergence rate can be obtained for any $N \in \mathbb{N}$ using an idea from \cite{CS02}.
\begin{corollary}\label{cor2}
For each $s \ge 1$ and $N \ge 2$ one can explicitly construct a point set $\widehat{\cP}_{N,s} \subset [0,1)^s$ such that
\begin{equation*}
\mathcal{L}_q(\widehat{\cP}_{N,s}) \ll_{s,q} \frac{(\log N)^{(s-1)/2}}{N} \quad \mbox{for all } 2 \le q < \infty.
\end{equation*}
\end{corollary}

\begin{proof}
It suffices to prove the result for all even integers $q$ with $2 \le q < \infty$. For given $N \ge 2$ choose $m \ge 1$ such that $2^{m-1} \le N < 2^m$. Then $\frac{2^m}{N} \le 2$. Let $\cP_{N,s}$ be the point set given by \eqref{Npoints}. It is elementary to check that the projection of $\mathscr{D}(\cP_{2^m, 2s})$ onto the first coordinate yields a point set which has exactly one point in each interval $[a 2^{-m}, (a+1) 2^{-m})$ for $0 \le a < 2^m$. This also follows from \cite[Proposition~1]{D10}. Thus $\cP_{N,s}$ contains exactly $N$ points.

Let $A([\bszero, \bstheta), N, \widehat{\cP}_{N,s}) = \sum_{n=0}^{N-1} 1_{[\bszero, \bstheta)}(\bsx_n)$ and let $\widetilde{\cP}_{N,s}$ be the point set given by \eqref{P_tilde}. Then we have
\begin{align*}
& (N \mathcal{L}_q(\widehat{\cP}_{N,s}))^q =  \int_{[0,1]^s} |\delta(\widehat{\cP}_{N,s}; \bstheta)|^q \,\mathrm{d} \bstheta \\ = & \int_{[0,1]^s} \left| A([0, N 2^{-m} \theta_1) \times \prod_{j=2}^s [0, \theta_j), N, \widetilde{\cP}_{N,s})  - 2^m \frac{N}{2^m} \theta_1 \theta_2 \cdots \theta_s \right|^q \,\mathrm{d} \bstheta \\  = & \frac{2^m}{N} \int_{0}^{N2^{-m}} \int_{[0,1]^{s-1}} \left| A([0,\bstheta), N, \widetilde{\cP}_{N,s}) - 2^m \theta_1 \theta_2 \cdots \theta_s\right|^q \,\mathrm{d} \bstheta \\ = & \frac{2^m}{N} \int_{0}^{N2^{-m}} \int_{[0,1]^{s-1}} \left| A([0,\bstheta), N, \mathscr{D}(\cP_{N,2s})) - 2^m \theta_1 \theta_2 \cdots \theta_s\right|^q \,\mathrm{d} \bstheta \\ \le & \frac{2^m}{N} (2^m \mathcal{L}_q(\mathscr{D}(\cP_{2^m, 2s})))^q.
\end{align*}
Thus we obtain $$\mathcal{L}_q(\widehat{\cP}_{N,s}) \le \left(\frac{2^m}{N} \right)^{1+1/q} \mathcal{L}_q(\mathscr{D}(\cP_{2^m,2s}) \le 3 \mathcal{L}_q(\mathscr{D}(\cP_{2^m,2s})) $$ and therefore
\begin{align*}
\mathcal{L}_q(\widehat{\cP}_{N,s}) \ll_{s,q} & \frac{m^{(s-1)/2}}{N}  \ll_{s,q} \frac{(\log N)^{(s-1) /2}}{N}.
\end{align*}
\end{proof}

The proof of Theorem~\ref{thm2} is presented in the next section.

\section{The proof of Theorem~\ref{thm2}}\label{sec_appendix}

The main analytical tool to prove the bound in Theorem~\ref{thm2} are Walsh functions. These are introduced in next subsection. The Walsh series expansion of the local discrepancy function is given in Subsection~\ref{Walsh_discrepancy}. Finally, the proof of Theorem~\ref{thm2} is presented in Subsection~\ref{subsec_proof}.

\subsection{Walsh functions and some of their properties}\label{sect_walsh}

In this section we introduce Walsh functions in base $2$ (see \cite{chrest,walsh}).

\begin{definition}\rm
For a non-negative integer $k$ with dyadic expansion
\[
   k = \kappa_{a-1} 2^{a-1} + \cdots + \kappa_1 2 + \kappa_0,
\]
with $\kappa_i \in \{0,1\}$ and $x \in [0,1)$ with dyadic expansion $$x =
\frac{x_1}{2}+\frac{x_2}{2^2}+\cdots $$ (unique in the sense that
infinitely many of the $x_i$ must be zero), we define the Walsh function
$\wal_{k}:[0,1) \rightarrow \{-1,1\}$ by
\[
  \wal_{k}(x) := (-1)^{x_1 \kappa_0 + \cdots + x_a \kappa_{a-1}}.
\]
\end{definition}

\begin{definition}\rm
For dimension $s \ge 2$, $\bsx = (x_1, \ldots, x_s) \in [0,1)^s$ and $\bsk = (k_1,
\ldots, k_s) \in \nat_0^s$ we define $\wal_{\bsk} : [0,1)^s
\rightarrow \{-1,1\}$ by
\[
   \wal_{\bsk}(\bsx) := \prod_{j=1}^s \wal_{k_j}(x_j).
\]
\end{definition}
Walsh functions are orthogonal in $\mathcal{L}_2$, that is, for any $\bsk, \bsell \in \mathbb{N}_0^s$ we have
\begin{equation}\label{walsh_orthogonal}
\int_{[0,1]^s} \wal_{\bsk}(\bsx) \wal_{\bsell}(\bsx) \rd \bsx = \left\{\begin{array}{rl} 1 & \mbox{if } \bsk = \bsell, \\ 0 & \mbox{otherwise}. \end{array} \right.
\end{equation}
Further, they are characters with respect to digital nets. That is, let $\cP_{2^m,s}$ be a digital net with generating matrices $C_1,\ldots, C_s$, then (cf. \cite[Lemma~4.75]{DP10})
\begin{equation}\label{char_prop}
\frac{1}{2^m} \sum_{n=0}^{2^m-1} \wal_{\bsk}(\bsx_n) = \left\{\begin{array}{rl} 1 & \mbox{if }  \bsk \in \mathcal{D}(C_1,\ldots, C_s), \\ 0 & \mbox{otherwise}. \end{array} \right.
\end{equation}

The classical Walsh functions were first used in earlier investigations of discrepancy in \cite{CS00} and in the related contexts of numerical integration in \cite{LT94} and pseudo random numbers in \cite{Tez2}. For more properties of Walsh functions see \cite{chrest,walsh}, for Walsh functions in the context of discrepancy see for instance \cite{CS02,DP05b,lp,Skr}, or \cite[Appendix~A]{DP10} in the context of numerical integration.

In the following we also introduce an inequality from \cite{Skr}, which is based on the Littlewood-Paley inequality for the Walsh function system and Minkowski's inequality. This inequality plays a central role in the proof of Theorem~\ref{thm2}.  

The following proposition is \cite[Lemma~4.2]{Skr}.
\begin{proposition}\label{prop_skr}
For $\bsb = (b_1,\ldots, b_s) \in \mathbb{N}_0^s$ we set $$B(\bsb) = \{\bsell \in \mathbb{N}_0^s: \mu_1(\ell_j) = b_j \mbox{ for } 1 \le j \le s\}.$$

Let $2 \le q < \infty$. For functions $f \in \mathcal{L}_q([0,1]^s)$ and $\bsb \in \mathbb{N}_0^s$ let
\begin{equation*}
\sigma_{\bsb} f(\bstheta) = \sum_{\bsell \in B(\bsb)} \widehat{f}(\bsell) \wal_{\bsell}(\bstheta)
\end{equation*}
where $\widehat{f}(\bsell) = \int_{[0,1]^s} f(\bsx) \wal_{\bsell}(\bsx) \rd \bsx$ is the $\bsell$th Walsh coefficient of $f$. Then for any $f \in \mathcal{L}_q([0,1]^s)$ we have
\begin{equation*}
\left(\int_{[0,1]^s} |f(\bstheta)|^q \,\mathrm{d} \bstheta \right)^{1/q} \ll_{q,s} \left(\sum_{\bsb \in \mathbb{N}_0^s} \left(\int_{[0,1]^s} |\sigma_{\bsb}f(\bstheta)|^q \,\mathrm{d} \bstheta \right)^{2/q} \right)^{1/2}.
\end{equation*}
\end{proposition}

\subsection{The Walsh series expansion of the $\mathcal{L}_q$ discrepancy function}\label{Walsh_discrepancy}

We now obtain the Walsh series expansion for the local discrepancy function. In the following the symbol `$\sim$' shall denote equality in the $\mathcal{L}_2$ norm sense. It is only used to point out which function corresponds to a given Walsh series.

We need the following notation. For $a \in \mathbb{R}$ let
\begin{equation*}
1_{a\neq 0} = \left\{\begin{array}{rl} 1 & \mbox{if } a \neq 0, \\ 0 & \mbox{if } a = 0. \end{array} \right.
\end{equation*}
For $\bsa = (a_1,\ldots, a_s)$, let $|\bsa|_1 = |a_1| + \cdots + |a_s|$, $1_{\bsa \neq \bszero} = (1_{a_1\neq 0}, \ldots, 1_{a_s\neq 0})$, $|1_{\bsa \neq \bszero}|_1 = \sum_{j=1}^s 1_{a_j\neq 0}$, and for a subset $u \subseteq \{1,\ldots, s\}$ let $\bsa_u = (a_j)_{j \in u}$, and $(\bsa_u, \bszero)$ denote the vector whose $j$th component is $a_j$ for $j \in u$ and $0$ otherwise. Let $k \in \mathbb{N}$ have dyadic expansion $k = \kappa_0 + \kappa_1 2 + \cdots + \kappa_{a-2} 2^{a-2} + 2^{a-1}$ with $\kappa_i \in \{0,1\}$. Further let $\bsk=(k_1,\ldots, k_s) \in \mathbb{N}_0^s$, $\nu(\bsk) = (\mu_1(k_1), \ldots, \mu_1(k_s))$, where $\mu_1$ is given by \eqref{def_mu}, and  $$\bsk \oplus \lfloor 2^{\bsa + \nu(\bsk) - \bsone} \rfloor = (k_1 \oplus \lfloor 2^{a_1+\mu_1(k_1)-1} \rfloor, \ldots,  k_s \oplus \lfloor 2^{a_s+\mu_1(k_s)-1} \rfloor).$$

\begin{lemma}\label{lem1}
The local discrepancy function has Walsh series expansion
\begin{align*}
& \delta(\cP_{N,s}; \bstheta) \\ \sim &  \frac{1}{2^s N} \sum_{n=0}^{N-1} \sum_{\bsk \in \mathbb{N}_0^s \setminus \{\bszero\}} 2^{-\mu_1(\bsk)}  \wal_{\bsk}(\bsx_n) \sum_{\bsa \in \mathbb{N}_0^s} (-1)^{|1_{\bsa \neq \bszero}|_1} 2^{-|\bsa|_1} \wal_{\bsk \oplus \lfloor 2^{\bsa + \nu(\bsk)-\bsone} \rfloor} (\bstheta).
\end{align*}
\end{lemma}

\begin{proof}
It is well known that (cf. \cite{Fine} or \cite[Lemma~A.22]{DP10})
\begin{equation*}
\theta = \sum_{a=0}^\infty (-1)^{1_{a \neq 0}} 2^{-a-1} \wal_{\lfloor 2^{a-1} \rfloor }(\theta).
\end{equation*}
The Walsh series expansion of the indicator function $1_{[0,t)}(x)$ can be obtained from \cite[Section~3]{Fine} (or \cite[Lemma~2]{lp} and \cite[Lemma~3]{lp}, or see also \cite[Lemma~14.8]{DP10}) and is given by
\begin{align*}
1_{[0,\theta)}(x)  \sim &  \sum_{a=0}^\infty  \sum_{k=0}^\infty  (-1)^{1_{a\neq 0}} 2^{-a-1} 2^{-\mu_1(k)} \wal_k(x)  \wal_{k \oplus \lfloor 2^{a+ \mu_1(k)-1} \rfloor}(\theta).
\end{align*}
By substituting these formulae into
\begin{align*}
\delta(\cP_{2^m,s}; \bstheta) = & \frac{1}{N} \sum_{n=0}^{N-1} \left( \prod_{j=1}^s 1_{[0, \theta_j)}(x_{n,j}) - \prod_{j=1}^s \theta_j \right)
\end{align*}
we obtain the result.
\end{proof}

In the following section we show how Definition~\ref{def_net} and Lemma~\ref{lem1} can be combined to obtain Theorem~\ref{thm2}.

\subsection{The proof of Theorem~\ref{thm2}}\label{subsec_proof}

Throughout this subsection we assume that $q$ is an even integer with $2 \le q < \infty$. Let $\cP_{N,s} = \{\bsx_0, \bsx_1, \ldots, \bsx_{N-1}\}$ denote the first $N$ points of an order $2$ digital $(t,s)$-sequence. We split the proof of Theorem~\ref{thm2} into several lemmas.
\begin{lemma}\label{lem_1}
For $N, s \in \mathbb{N}$, with $N \ge 2$, we have
\begin{align*}
\mathcal{L}^2_q(\cP_{N,s}) \ll_{q,s} \sum_{\bsb \in \mathbb{N}_0^s} \left(\int_{[0,1]^s} |\sigma_{\bsb} (\bstheta)|^q \,\mathrm{d} \bstheta \right)^{2/q},
\end{align*}
where
\begin{align*}
\sigma_{\bsb}(\bstheta) = \sum_{\bsell \in B(\bsb)} c(\bsell) \wal_{\bsell}(\bstheta),
\end{align*}
and where
\begin{align*}
c(\bsell) = & \sum_{u \in A(\bsell)} (-1)^{s-|u|} \sum_{\bsz_{u} \in \mathbb{N}^{|u|} } 2^{-\mu_1(\bsell) - |\bsz_{u}|_1 - s} \frac{1}{N} \sum_{n=0}^{N-1} \wal_{\bsell \oplus \lfloor 2^{(\bsz_{u}, \bszero) + \nu(\bsell) - \bsone}\rfloor}(\bsx_n),
\end{align*}
where the set $A(\bsell)$ is the power set $\mathscr{P}(\{1,\ldots, s\})$ of $\{1,\ldots, s\}$, unless $\bsell = (\ell_1,\ldots, \ell_s)$ is of the form $\ell_j = 0$ or $\ell_j = 2^{\mu_1(\ell_j)-1}$ for all $1 \le j \le s$, in which case the term $u = \emptyset$ is omitted, i.e.
\begin{equation*}
A(\bsell) = \left\{\begin{array}{rl} \mathscr{P}(\{1,\ldots, s\}) \setminus \{\emptyset\} & \mbox{if } \ell_j = 0 \mbox{ or } \ell_j = 2^{\mu_1(\ell_j)-1} \mbox{ for } 1 \le j \le s, \\ \mathscr{P}(\{1,\ldots, s\}) & \mbox{otherwise}. \end{array} \right.
\end{equation*}
\end{lemma}

\begin{proof}
Using Lemma~\ref{lem1} we obtain the following expression for the Walsh series of the local discrepancy function $\delta(\cP_{N,s}; \bstheta)$
\begin{align}\label{big_sum}
& \frac{1}{2^sN} \sum_{n=0}^{N-1} \sum_{\bsk\in \mathbb{N}_0^s \setminus \{\bszero\} }  2^{-\mu_1(\bsk)} \wal_{\bsk}(\bsx_n)\sum_{v \subseteq \{1,\ldots, s\}} (-1)^{|v|} \sum_{\bsa_v \in \mathbb{N}^{|v|}}  2^{-|\bsa_v|_1} \wal_{\bsk \oplus \lfloor 2^{(\bsa_v,\bszero) + \nu(\bsk)-\bsone} \rfloor }(\bstheta).
\end{align}

The sum \eqref{big_sum} can be rearranged using the substitution $\ell_j = k_j \oplus 2^{a_j + \mu_1(k_j)-1}$. Thus we can write $k_j = \ell_j \oplus 2^{z_j + \mu_1(\ell_j)-1}$. In this case we have 
\begin{align*}
\mu_1(\ell_j) = \mu_1(k_j) + a_j & \mbox{ if } a_j > 0, \\ 
\mu_1(k_j) = \mu_1(\ell_j) + z_j & \mbox{ if } z_j > 0.
\end{align*}
The two cases are complementary, that is, if $a_j > 0$, then $\mu_1(\ell_j) > \mu_1(k_j)$ and then $z_j  = 0$. By symmetry we also have that if $z_j > 0$, then $a_j = 0$.

In \eqref{big_sum} we sum over all $v \subseteq \{1,\ldots, s\}$ and then sum over the vector $\bsa_v$ which consists of positive integers. By the substitution above we replace $k_j$ by $\ell_j$ and $a_j$ by $z_j$, where $z_j = 0$ for $a_j > 0$ and $z_j > 0$ for $a_j = 0$. Thus we need to replace the sum over $v \subseteq \{1,\ldots, s\}$ by a sum over $u := \{1,\ldots, s\} \setminus v$. Hence we obtain
\begin{align}\label{big_sum2}
& \frac{1}{2^sN} \sum_{n=0}^{N-1} \sum_{\bsk\in \mathbb{N}_0^s \setminus \{\bszero\} }  2^{-\mu_1(\bsk)} \wal_{\bsk}(\bsx_n) \\ & \sum_{v \subseteq \{1,\ldots, s\}} (-1)^{|v|} \sum_{\bsa_v \in \mathbb{N}^{|v|}}  2^{-|\bsa_v|_1} \wal_{\bsk \oplus \lfloor 2^{(\bsa_v,\bszero) + \nu(\bsk)-\bsone} \rfloor }(\bstheta) \nonumber \\ = & \frac{1}{2^sN} \sum_{n=0}^{N-1} \sum_{u \subseteq \{1,\ldots, s\}} (-1)^{s-|u|} \sum_{\bsz_u \in \mathbb{N}^{|u|}} 2^{-|\bsz_u|_1} \nonumber \\ &  \sum_{ \satop{ \bsell \in \mathbb{N}_0^s }{\bsell \oplus \lfloor 2^{\bsz_u + \mu_1(\bsell)- \bsone } \rfloor \neq \bszero  } } 2^{-\mu_1(\bsell)}  \wal_{\bsell \oplus \lfloor 2^{\bsz_u + \mu_1(\bsell)- \bsone } \rfloor  }(\bsx_n)  \wal_{\bsell}(\bstheta).
\end{align}

Note that for $k_j = 2^{\mu_1(k_j)-1}$ and $a_j = 0$ we have $\ell_j = 0$. Thus the case  $\bsell = \bszero$ needs to be included. On the other hand, we must have $\bsell \oplus \lfloor 2^{(\bsz_u,\bszero) + \nu(\bsell) - \bsone} \rfloor \neq \bszero$. We have $\ell_j \oplus \lfloor 2^{z_j+\mu_1(\ell_j)-1} \rfloor = 0$ if either $\ell_j = 0$ and $z_j=0$ or $\ell_j = 2^{\mu_1(\ell_j)-1}$ and $z_j=0$. We can exclude this case by assuming that if all $\ell_j$ are of the form $\ell_j = 0$ or $\ell_j = 2^{\mu_1(\ell_j)-1}$ then $u \neq \emptyset$, since then at least one $z_j \neq 0$. This is ensured by the condition on the sum over $u$ in the statement of the lemma. Thus we obtain that \eqref{big_sum2} equals
\begin{align*}
\sum_{\bsell \in \mathbb{N}_0^s} c(\bsell) \wal_{\bsell}(\bstheta)  = & \sum_{\bsb \in \mathbb{N}_0^s} \sum_{\bsell \in B(\bsb)} c(\bsell) \wal_{\bsell}(\bstheta)
\end{align*}
and we have
\begin{equation}\label{eq_c}
\delta(\cP_{N,s}; \bstheta) \sim \sum_{\bsb \in \mathbb{N}_0^s} \sigma_{\bsb}(\bstheta).
\end{equation}

Using Proposition~\ref{prop_skr} applied to the local discrepancy function $\delta(\cP_{N,s}; \cdot)$ we obtain the result.
\end{proof}

Using the orthogonality of the Walsh functions \eqref{walsh_orthogonal} and the assumption that $q \ge 2$ is an even integer, the integral over $|\sigma_{\bsb}|^q$ can be written in terms of the coefficients $c(\bsell_i)$ in the following way:
\begin{align}\label{bound_sumb}
\int_{[0,1]^s} |\sigma_{\bsb}(\bstheta) |^{q} \,\mathrm{d} \bstheta = & \sum_{\bsell_1, \ldots, \bsell_{q} \in B(\bsb)} \prod_{i=1}^{q} c(\bsell_i) \int_{[0,1]^s}  \prod_{i=1}^{q} \wal_{\bsell_i}(\bstheta)  \,\mathrm{d} \bstheta \nonumber \\ = & \sum_{\bsell_1,\ldots, \bsell_{q-1} \in B(\bsb)} c(\bsell_1 \oplus \cdots \oplus \bsell_{q-1}) \prod_{i=1}^{q-1} c(\bsell_i).
\end{align}
We now obtain a bound on $\int_{[0,1]^s} |\sigma_{\bsb}(\bstheta) |^{q} \,\mathrm{d} \bstheta$ by bounding the coefficients $c(\bsell_i)$ and $c(\bsell_1 \oplus \cdots \oplus \bsell_{q-1})$.

\begin{lemma}
Let $N \in \mathbb{N}$ with $N \ge 2$ and $N = 2^{m_1} + 2^{m_2} + \cdots + 2^{m_r}$ with $m_1 > m_2 > \cdots > m_r \ge 0$ and $\bsell \in B(\bsb)$. Then we have
\begin{align}\label{cl_est}
|c(\bsell)| \ll_s &  \sum_{h=1}^r \frac{2^{m_h}}{N} \sum_{\satop{ \bsz \in \mathbb{N}_0^{s}}{ \bsell \oplus \lfloor 2^{\bsz + \nu(\bsell) - \bsone } \rfloor  \in \mathcal{D}_{m_h,s}^\ast}} 2^{- |\bsb|_1 - |\bsz|_1 }.
\end{align}
\end{lemma}

\begin{proof}
In the following we rewrite the expression for $c(\bsell)$ from Lemma~\ref{lem_1} which will make it easier to obtain a suitable bound. To do so, we need the following notation. Let $C_1, \ldots, C_s \in \mathbb{F}_2^{\mathbb{N} \times \mathbb{N}}$ denote the generating matrices of $\cS_s$. Let $C_{j}^{(2m_h \times m_h)}$ denote the left upper submatrix of $C_j$ of size $2m_h \times m_h$. We divide the point set $\{\bsx_0, \bsx_1, \ldots, \bsx_{N-1}\}$ into blocks of size $2^{m_h}$ in the following way: Let
\begin{equation*}
Q_h = \{\bsx_{2^{m_1}+\cdots + 2^{m_{h-1}}}, \bsx_{2^{m_1}+\cdots + 2^{m_{h-1}} + 1}, \ldots, \bsx_{2^{m_1} + \cdots + 2^{m_{h}}-1}\},
\end{equation*}
for $1 \le h \le r$, where for $h=1$ we set $2^{m_1} + \cdots + 2^{m_{h-1}} = 0$.

The main reason for dividing the point set in this way is that $Q_h$ is a digitally shifted digital net over $\mathbb{F}_2$ with generating matrices $C_1^{2m_h \times m_h}, \ldots, C_s^{2m_h \times m_h}$, where the digital shift is done by dyadic rationals, see Lemma~\ref{lem_net_seq}. Assume that the digital shift is given by $\sigma_h$. We have $Q_h \oplus \bssigma_h =  \{\bsx \oplus \bssigma_h: \bsx \in Q_h\}$ is a digital net with generating matrices $C_1^{(2m_h \times m_h)}, \ldots, C_s^{(2m_h \times m_h)}$ (note that for $\mathbb{F}_2$ the symbol $\oplus$ coincides with $\ominus$). Let $\mathcal{D}_{m_h,s}$ denote the dual net corresponding to $Q_h$, that is, $\mathcal{D}_{m_h,s} = \mathcal{D}(C_1^{(2m_h \times m_h)}, \ldots, C_s^{(2m_h \times m_h)})$. Further $$\wal_{\bsell \oplus \lfloor 2^{(\bsz_{u}, \bszero) + \nu(\bsell) - \bsone}\rfloor}(\bsx_n) = \wal_{\bsell \oplus \lfloor 2^{(\bsz_{u}, \bszero) + \nu(\bsell) - \bsone}\rfloor}(\bsx'_n) \wal_{\bsell \oplus \lfloor 2^{(\bsz_{u}, \bszero) + \nu(\
bsell) - \bsone}\rfloor}(\bssigma_h),$$ where $\bsx_n = \bsx'_n \oplus \bssigma_
h$ (note that all components of $\bsx_n, \bssigma, \bsx'_n$ are dyadic rationals). We can use the character property \eqref{char_prop} for the digital net $Q_h \oplus \bssigma_h$ to obtain
\begin{align*}
c(\bsell) = & \sum_{h=1}^r \frac{2^{m_h}}{N} \sum_{u\subseteq \{1,\ldots, s\}}^\ast (-1)^{s-|u|} \hspace{-1cm} \sum_{\satop{ \bsz_{u} \in \mathbb{N}^{|u|}}{ \bsell \oplus \lfloor 2^{(\bsz_u, \bszero) + \nu(\bsell) - \bsone } \rfloor  \in \mathcal{D}_{m_h,s}^\ast}} \hspace{-1cm} 2^{-\mu_1(\bsell) - |\bsz_{u}|_1 - s} \wal_{\bsell \oplus \lfloor 2^{(\bsz_{u}, \bszero) + \nu(\bsell) - \bsone} \rfloor}(\bssigma_h).
\end{align*}

We now estimate $|c(\bsell)|$ for $\bsell \in B(\bsb)$. Using the facts that $|(-1)^{s-|u|}|=1$, $|\wal_{\bsell \oplus \lfloor 2^{(\bsz_{u}, \bszero) + \nu(\bsell) - \bsone} \rfloor}(\bssigma_h)| = 1$ and the triangle inequality we deduce
\begin{align*}
|c(\bsell)| \ll_s &  \sum_{h=1}^r \frac{2^{m_h}}{N} \sum_{u \subseteq \{1,\ldots, s\}}  \sum_{\satop{ \bsz_u \in \mathbb{N}^{|u|}}{ \bsell \oplus \lfloor 2^{(\bsz_u, \bszero) + \nu(\bsell) - \bsone } \rfloor  \in \mathcal{D}_{m_h,s}^\ast}} 2^{-\mu_1(\bsell) - |\bsz_u|_1 } \\ = & \sum_{h=1}^r \frac{2^{m_h}}{N} \sum_{\satop{ \bsz \in \mathbb{N}_0^{s}}{ \bsell \oplus \lfloor 2^{\bsz + \nu(\bsell) - \bsone } \rfloor  \in \mathcal{D}_{m_h,s}^\ast}} 2^{-\mu_1(\bsell) - |\bsz|_1 }.
\end{align*}
Since $\mu_1(\bsell) = |\bsb|_1$ we obtain the result.
\end{proof}

The following lemma proves an effective bound on $|c(\bsell)|$. In the proof of this result we make essential use of the order $2$ digital net property of our point set.

\begin{lemma}\label{lem_c_bound}
Let $\bsell \in B(\bsb)$ and let $N \in \mathbb{N}$ have dyadic expansion $N = 2^{m_1} + 2^{m_2} + \cdots + 2^{m_r}$ with $m_1 > m_2 > \cdots > m_r \ge 0$.  Then we have
\begin{equation*}
|c(\bsell)| \ll_s \frac{1}{N} \sum_{h=1}^r 2^{m_h - |\bsb|_1 -2(m_h-|\bsb|_1)_+} {2(m_h-|\bsb|_1)_+ + s-1 \choose s-1},
\end{equation*}
where $(v)_+ = \max\{0, v\}$.
\end{lemma}

\begin{proof}
Let $\bsell = (\ell_{1},\ldots, \ell_{s})$ and $\ell_{j} = 2^{w_{j,1}-1} + 2^{w_{j,2}-1} + \cdots + 2^{w_{j,r_{j}}-1}$, where $w_{j,1} > w_{j,2} > \cdots > w_{j,r_{j}} > 0$ for $\ell_j > 0$ and for $\ell_j=0$ we set $w_{j,1}= w_{j,2} =0$ and $r_{j}=0$. Further we set $w_{j, r_j+1}  = w_{j, r_j+2} = w_{j, r_j+3} = 0$. For $\bsz = (z_1, \ldots, z_s) \in \mathbb{N}_0^s$ let $u$ be the set of components $j$ for which $z_j > 0$. Using the definition of $\mu_2$ given by \eqref{def_mu2} and \eqref{def_mu2_vec}, we have
\begin{equation*}
\mu_2(\bsell \oplus \lfloor 2^{\bsz + \nu(\bsell)-\bsone} \rfloor) = \sum_{j\in u} (2 \mu_1(\ell_{j}) + z_{j}) + \sum_{j \in \{1,\ldots, s\} \setminus u} (w_{j,2}+w_{j,3}).
\end{equation*}
By the order $2$ digital net property, in particular \eqref{2min_weight}, and $\bsell \oplus \lfloor 2^{\bsz + \nu(\bsell)-\bsone} \rfloor \in \mathcal{D}_{m_h,s}^\ast$ for $\bsz \in \mathbb{N}_0^{s}$ we have
\begin{align*}
2 m_h -t+1 \le & \mu_2(\bsell \oplus \lfloor 2^{\bsz + \nu(\bsell)-\bsone} \rfloor) \\ = & \sum_{j\in u} (2 \mu_1(\ell_{j}) + z_{j}) + \sum_{j \in \{1, \ldots, s\} \setminus u} (w_{j,2}+w_{j,3}) \\ \le & 2\mu_1(\bsell) + |\bsz|_1.
\end{align*}
Since $\mu_1(\bsell) = |\bsb|_1$, we obtain
\begin{equation}\label{zb_bound2}
|\bsz|_1 \ge 2 m_h -  t + 1 - 2 |\bsb|_1.
\end{equation}
Using \eqref{zb_bound2}, the right-most sum in \eqref{cl_est} can be split into the cases where $2 m_h -  t + 1 - 2 |\bsb|_1 \le 0$ and where $2 m_h -  t + 1 - 2 |\bsb|_1 > 0$. If $2 m_h -  t + 1 - 2 |\bsb|_1 \le 0$ we sum over all $\bsz \in \mathbb{N}^{s}$, which is
\begin{equation*}
\sum_{\bsz \in \mathbb{N}_0^{s}} 2^{- |\bsz|_1} = \left(\sum_{z=0}^\infty 2^{-z} \right)^s = 2^s.
\end{equation*}

In the following we make use of the well-known inequality
\begin{equation}\label{ineq_el}
\sum_{a=a_0}^\infty b^{-a} {a + s-1 \choose s-1} \le b^{-a_0} {a_0 + s-1 \choose s-1} \left(1-\frac{1}{b}\right)^{-s}.
\end{equation}
A proof can for instance be found in \cite[Lemma~2.18]{mat}.

In the second case we have $2m_h-t+1-2|\bsb|_1 > 0$. From \eqref{2min_weight} we obtain that for $\bsz$ satisfying $\bsell \oplus \lfloor 2^{\bsz + \nu(\bsell) - \bsone } \rfloor  \in \mathcal{D}_{m_h,s}^\ast$, we have $|\bsz|_1 \ge 2m_h-t+1-2|\bsb|_1$. Thus
\begin{align*}
\sum_{\satop{ \bsz \in \mathbb{N}_0^{s}}{ \bsell \oplus \lfloor 2^{\bsz + \nu(\bsell) - \bsone } \rfloor  \in \mathcal{D}_{m_h,s}^\ast}} 2^{- |\bsz|_1 } \le & \sum_{\satop{ \bsz \in \mathbb{N}_0^{s}}{|\bsz|_1 \ge 2m_h-t+1-2|\bsb|_1 }} 2^{- |\bsz|_1 } \\ = & \sum_{a= 2m_h-t+1 - 2|\bsb|_1}^\infty 2^{-a} {a + s-1 \choose s-1} \\ \ll_s & 2^{-2m_h + 2|\bsb|_1} {2m_h-2|\bsb|_1 + s-1 \choose s-1}.
\end{align*}
Thus we obtain
\begin{align*}
|c(\bsell)| \ll_s &  \frac{1}{N} \sum_{h=1}^r 2^{m_h - |\bsb|_1} 2^{-2(m_h-|\bsb|_1)_+} {2(m_h-|\bsb|_1)_+ + s-1 \choose s-1},
\end{align*}
where we used that $t$ depends only on the dimension $s$.
\end{proof}

We can also estimate $|c(\bsell_1 \oplus \cdots \oplus \bsell_{q-1})|$. Note that for $\bsell_1,\ldots, \bsell_{q-1} \in B(\bsb)$ and for $q$ even we have $\bsell_1 \oplus \cdots \oplus \bsell_{q-1} \in B(\bsb)$. Therefore, by Lemma~\ref{lem_c_bound} we have
\begin{equation}\label{cl_est2}
|c(\bsell_1 \oplus \cdots \oplus \bsell_{q-1})| \ll_s \frac{1}{N} \sum_{h=1}^r 2^{m_h - |\bsb|_1} 2^{-2(m_h-|\bsb|_1)_+} {2(m_h-|\bsb|_1)_+ + s-1 \choose s-1}.
\end{equation}

We return to the initial aim of bounding \eqref{bound_sumb}. Since the right-hand side of \eqref{cl_est2} depends only on $\bsb$ but is independent of $\bsell_1,\ldots, \bsell_{q-1}$, we obtain a bound on \eqref{bound_sumb}
\begin{align}\label{expr_parenthesis}
& \left|\sum_{\bsell_1,\ldots, \bsell_{q-1} \in B(\bsb)} c(\bsell_1 \oplus \cdots \oplus \bsell_{q-1}) \prod_{i=1}^{q-1} c(\bsell_i) \right| \nonumber \\ \le &  \sum_{\bsell_1,\ldots, \bsell_{q-1} \in B(\bsb)} |c(\bsell_1 \oplus \cdots \oplus \bsell_{q-1})| \prod_{i=1}^{q-1} |c(\bsell_i)| \nonumber \\ \ll_s & \frac{1}{N} \sum_{h=1}^r 2^{m_h - |\bsb|_1} 2^{-2(m_h-|\bsb|_1)_+} {2(m_h-|\bsb|_1)_+ + s-1 \choose s-1} \left(\sum_{\bsell \in B(\bsb)} |c(\bsell)| \right)^{q-1} \nonumber  \\ \ll_s & \frac{1}{N^q} \sum_{h=1}^r 2^{m_h - |\bsb|_1 -2(m_h-|\bsb|_1)_+} {2(m_h-|\bsb|_1)_+ + s-1 \choose s-1} \nonumber \\  & \left(\sum_{h=1}^r 2^{m_h - |\bsb|_1} \sum_{\bsell \in B(\bsb)} \sum_{\satop{ \bsz \in \mathbb{N}_0^{s}}{ \bsell \oplus \lfloor 2^{\bsz + \nu(\bsell) - \bsone } \rfloor  \in \mathcal{D}_{m_h,s}^\ast}} 2^{ - |\bsz|_1 } \right)^{q-1}.
\end{align}

The aim is now to obtain a bound on the expression in parenthesis in \eqref{expr_parenthesis}. We prove an auxiliary result first.

\begin{lemma}\label{lem_aux}
For $|\bsb|_1 \ge m_h$ we have
\begin{equation*}
\sum_{\bsell \in B(\bsb)} \sum_{\satop{ \bsz \in \mathbb{N}_0^{s}}{ \bsell \oplus \lfloor 2^{\bsz + \nu(\bsell) - \bsone } \rfloor  \in \mathcal{D}_{m_h,s}^\ast}} 2^{- |\bsz|_1 } \ll_s 2^{|\bsb|_1-m_h}
\end{equation*}
and for $|\bsb|_1 < m_h$ we have
\begin{equation*}
\sum_{\bsell \in B(\bsb)} \sum_{\satop{ \bsz \in \mathbb{N}_0^{s}}{ \bsell \oplus \lfloor 2^{\bsz + \nu(\bsell) - \bsone } \rfloor  \in \mathcal{D}_{m_h,s}^\ast}} 2^{- |\bsz|_1 } \ll_s 2^{-2m_h + 2|\bsb|_1} {2m_h-2|\bsb|_1 + s \choose s-1}.
\end{equation*}
\end{lemma}

\begin{proof}
Let $\bsz \in \mathbb{N}_0^s$ be fixed. We count the number of $\bsell \in B(\bsb)$ such that $\lfloor \bsell \oplus 2^{\bsz + \nu(\bsell) - \bsone} \rfloor \in \mathcal{D}_{m_h,s}^\ast$. Let $C_1^{(2m_h \times m_h)},\ldots, C_s^{(2m_h \times m_h)}$ denote the generating matrices of the digital net and let $\vec{c}^{(h)}_{j,k}$ denote the $k$th row of $C_j^{(2m_h \times m_h)}$ for $1 \le k \le 2m_h$ and let $\vec{c}^{(h)}_{j,k} = \vec{0}$ for $k > 2m_h$. Let $\bsell = (\ell_{1},\ldots, \ell_{s})$. Assume that $\ell_{j} = 2^{w_{j,1}-1} + 2^{w_{j,2}-1} + \cdots + 2^{w_{j,r_{j}}-1}$, where $w_{j,1} > w_{j,2} > \cdots > w_{j,r_{j}} > 0$ and also that $\ell_{j} = \ell_{j,0} + 2 \ell_{j,1} + \cdots + \ell_{j,w_{j,1}-1} 2^{w_{j,1}-1}$, where $\ell_{j, w_{j,1}-1} = 1$. The condition $\lfloor \bsell \oplus 2^{\bsz + \nu(\bsell) - \bsone} \rfloor \in \mathcal{D}_{m_h,s}^\ast$ translates into the system of equations
\begin{eqnarray}\label{eq_sys}
\vec{c}_{1,1}^{(h)} \ell_{1,0}+\cdots +\vec{c}_{1,w_{1,1}-1}^{(h)}
\ell_{1,w_{1,1}-2}+  \vec{c}_{1,w_{1,1}}^{(h)} + && \nonumber \\
\vec{c}_{2,1}^{(h)} \ell_{2,0}+\cdots + \vec{c}_{2,w_{2,1}-1}^{(h)}
\ell_{2,w_{2,1} -2}+ \vec{c}_{2,w_{2,1}}^{(h)} + &&  \nonumber \\
\vdots && \nonumber \\
\vec{c}_{s,1}^{(h)} \ell_{s,0}+\cdots +\vec{c}_{s,w_{s,1}-1}^{(h)}
\ell_{s,w_{s,1}-2} + \vec{c}_{s,w_{s,1}}^{(h)} \hspace{0.3cm}& = & \vec{c}^{(h)},
\end{eqnarray}
where the vector $\vec{c}^{(h)}$ on the right hand side is fixed by $\bsz$. Consider now the set of vectors $\{ \vec{c}_{j,k}^{(h)}: 1 \le k < w_{j,1}, 1 \le j \le s\}$ from \eqref{eq_sys}. The minimum weight bound \eqref{min_weight} implies that at least $m_h - t + 1$ of those vectors are linearly independent. Thus at most $(\mu_1(\bsell) - m_h + t-1)_+$ of the $\ell_{j,k}$ can be chosen freely, the remaining ones need to be fixed in order to obtain a solution of \eqref{eq_sys}. Hence it follows that the number of solutions of the linear system \eqref{eq_sys} is at most
\begin{equation*}
2^{(\mu_1(\bsell) - m_h + t-1)_+} =  2^{(|\bsb|_1 - m_h + t-1)_+}  \le 2^{t-1} 2^{ (|\bsb|_1 - m_h)_+} \ll_s 2^{(|\bsb|_1 - m_h)_+},
\end{equation*}
where $(x)_+ = \max\{x, 0\}$ and where we used that $t$ depends only on $s$.  Using \eqref{zb_bound2} and the bound above we obtain
\begin{align*}
& \sum_{\bsell \in B(\bsb)} \sum_{\satop{\bsz \in \mathbb{N}_0^{s}}{ (\bsell \oplus 2^{\bsz + \nu(\bsell) - \bsone}) \in \mathcal{D}_{m_h,s}^\ast }} 2^{ -  |\bsz|_1} \ll_s  \sum_{\satop{\bsz\in \mathbb{N}_0^s}{|\bsz|_1 \ge 2m_h - t+1-2|\bsb|_1}} 2^{- |\bsz|_1 }  2^{(|\bsb|_1-m_h)_+}.
\end{align*}
If $|\bsb|_1 \ge m_h$, then the above sum is bounded by (using again that $t$ depends only on $s$)
\begin{equation}\label{bound_lsb1}
2^{|\bsb|_1 -m_h +t-1} \sum_{\bsz \in \mathbb{N}_0^s} 2^{-|\bsz|_1} = 2^{|\bsb|_1 -m_h+t-1+s} \ll_s 2^{|\bsb|_1 - m_h}.
\end{equation}
If $|\bsb|_1 < m_h$, then the above sum is bounded by
\begin{align}\label{bound_lsb2}
\sum_{\satop{\bsz\in \mathbb{N}_0^s}{|\bsz|_1 \ge 2m_h -t+1-2|\bsb|_1}} 2^{- |\bsz|_1}  \le & \sum_{a=2m_h -t+1-2|\bsb|_1}^\infty 2^{-a} {a+s-1\choose s-1} \nonumber \\ \ll_{s} & 2^{-2m_h  + 2|\bsb|_1} {2m_h -2|\bsb|_1 + s - t \choose s-1} \nonumber \\ \ll_s & 2^{-2m_h  + 2|\bsb|_1} {2m_h -2|\bsb|_1 + s \choose s-1},
\end{align}
where we set ${n \choose k} = 0$ for $n < k$ and where we again used that $t$ depends only on $s$.
\end{proof}

In the following lemma we show a bound on the expression in parenthesis in \eqref{expr_parenthesis}.
\begin{lemma}\label{lem_bound_sum}
Let $N \in \mathbb{N}$ with $N \ge 2$ and $N = 2^{m_1} + 2^{m_2} + \cdots + 2^{m_r}$, where $m_1 > m_2 > \cdots > m_r \ge 0$. Set $m_0 = \infty$ and $m_{r+1} = 0$. Then we have
\begin{align*}
\sum_{h=1}^r 2^{m_h - |\bsb|_1} \sum_{\bsell \in B(\bsb)} \sum_{\satop{ \bsz \in \mathbb{N}_0^{s}}{ \bsell \oplus \lfloor 2^{\bsz + \nu(\bsell) - \bsone } \rfloor  \in \mathcal{D}_{m_h,s}^\ast}} 2^{ - |\bsz|_1 } \ll_s r.
\end{align*}
\end{lemma}

\begin{proof}
For $\bsb \in \mathbb{N}_0^s$ let $1 \le h_0= h_0(|\bsb|_1) \le r+1$ be the integer which satisfies $m_{h_0-1} > |\bsb|_1 \ge m_{h_0}$.  Using Lemma~\ref{lem_aux} we obtain
\begin{align*}
& \sum_{h=1}^r 2^{m_h - |\bsb|_1} \sum_{\bsell \in B(\bsb)} \sum_{\satop{ \bsz \in \mathbb{N}_0^{s}}{ \bsell \oplus \lfloor 2^{\bsz + \nu(\bsell) - \bsone } \rfloor  \in \mathcal{D}_{m_h,s}^\ast}} 2^{ - |\bsz|_1 }  \\  \ll_s & \sum_{h=1}^{h_0-1} 2^{|\bsb|_1-m_h} {2m_h-2|\bsb|_1 + s \choose s-1} + \sum_{h=h_0}^{r}  1.
\end{align*}
We now estimate the sum over $1 \le h < h_0$, which is essentially a sum over $\{m_1 - |\bsb|_1, m_2 - |\bsb|_1, \ldots, m_{h_0(|\bsb|_1)} - |\bsb|_1\}$. We replace this set by $\mathbb{N}$, that is
\begin{equation*}
\sum_{h=1}^{h_0(|\bsb|_1)} 2^{|\bsb|-m_h} {2m_h-2|\bsb|_1 +s \choose s-1} \le \sum_{a=1}^\infty 2^{-a} {2a + s \choose s-1} \ll_s 1.
\end{equation*}
Thus the result follows.
\end{proof}

We can now obtain a bound on $\int |\sigma_{\bsb}|^q$.

\begin{lemma}\label{lem_bound_int_sigma}
Let $N \in \mathbb{N}$ with $N \ge 2$ have dyadic expansion $N = 2^{m_1} + 2^{m_2} + \cdots + 2^{m_r}$, where $m_1 > m_2 > \cdots > m_r \ge 0$. Then
\begin{align*}
\int_{[0,1]^s} |\sigma_{\bsb}(\bstheta) |^{q} \,\mathrm{d} \bstheta \ll_{s} & \frac{r^{q-1} }{N^q} \sum_{h=1}^r 2^{m_h-|\bsb|_1- 2(m_h-|\bsb|_1)_+} {2(m_h-|\bsb|_1)_+ + s-1 \choose s-1}.
\end{align*}
\end{lemma}

\begin{proof}
The result follows by combining \eqref{bound_sumb}, \eqref{expr_parenthesis} and Lemma~\ref{lem_bound_sum}.
\end{proof}

The following lemma completes the proof of Theorem~\ref{thm2}.

\begin{lemma}\label{lem_Lq_bound_weak}
For any $N \in \mathbb{N}$ with $N \ge 2$ and $N = 2^{m_1} +2^{m_2} + \cdots + 2^{m_r}$, where $m_1 > m_2 > \cdots > m_r \ge 0$, the first $N$ points $\cP_{N,s}$ of the sequence $\cS_s$ satisfy
\begin{equation*}
\mathcal{L}_q(\cP_{N,s}) \ll_{s, q} \frac{r^{3/2-1/q} }{N} \sqrt{\sum_{v=1}^r m_v^{s-1}},
\end{equation*}
for all even integers $q$ with $2 \le q < \infty$.
\end{lemma}

\begin{proof}
Using Lemma~\ref{lem_1} and Lemma~\ref{lem_bound_int_sigma} we obtain
\begin{align}\label{ineq_Lq_bound}
\mathcal{L}_q^2(\cP_{N,s}) \ll_{s} & c_{q,s} \frac{r^{2(1-1/q)} }{N^2} \sum_{\bsb \in \mathbb{N}_0^s} \left(\sum_{h=1}^r 2^{m_h-|\bsb|_1- 2(m_h-|\bsb|_1)_+} {2(m_h-|\bsb|_1)_+ + s-1 \choose s-1} \right)^{2/q}.
\end{align}

It remains to estimate the sum over $\bsb$. We have
\begin{align*}
& \sum_{\bsb \in \mathbb{N}_0^s} \left(\sum_{h=1}^r 2^{m_h-|\bsb|_1 - 2(m_h-|\bsb|_1)_+} {2(m_h-|\bsb|_1)_+ + s-1 \choose s-1} \right)^{2/q} \\ = & \sum_{a=0}^{\infty} {a+s-1 \choose s-1} \left( \sum_{h=1}^r 2^{m_h-a-2(m_h-a)_+} {2 (m_h-a)_+ + s-1 \choose s-1} \right)^{2/q}.
\end{align*}

We split the above sum into the part where $a \ge m_1$ and where $0 \le a < m_1$. For $a \ge m_1$ we have ${2(m_h-a)_+ + s-1 \choose s-1} = 1$ and
\begin{equation*}
\sum_{h=1}^r 2^{m_h-a- 2(m_h-a)_+} = \sum_{h=1}^r 2^{m_h-a} = 2^{-a} \sum_{h=1}^r 2^{m_h} = 2^{-a} N.
\end{equation*}
Thus we can use \eqref{ineq_el} to obtain
\begin{align*}
& \sum_{a=m_1}^{\infty} {a+s-1 \choose s-1} \left( \sum_{h=1}^r 2^{m_h-a-2(m_h-a)_+} {2 (m_h-a)_+ + s-1 \choose s-1} \right)^{2/q} \\ \le & \sum_{a=m_1 }^\infty N^{2/q} 2^{-2a/q} {a+s-1 \choose s-1} \\ \ll_s & m_1^{s-1}.
\end{align*}

For $0 \le a < m_1$ we use Jensen's inequality, which states that for a sequence of nonnegative real numbers $(a_j)$ and any real number $0 < \lambda \le 1$ we have $(\sum_j a_j)^\lambda \le \sum_j a_j^\lambda$. Since $2/q \le 1$ we have
\begin{align*}
& \sum_{a=0}^{m_1-1} {a+s-1 \choose s-1} \left( \sum_{h=1}^r 2^{m_h-a- 2(m_h-a)_+} {2(m_h-a)_+ + s-1 \choose s-1}  \right)^{2/q} \\ \le  &  \sum_{h=1}^r  \sum_{a=0}^{m_1-1} {a+s-1 \choose s-1} 2^{\frac{2}{q} [m_h-a- 2(m_h-a)_+ ] } {2(m_h-a)_+ + s-1 \choose s-1}^{2/q}.
\end{align*}
We split the sum over $a$ into the parts $m_{v+1} \le a < m_v$. Let $m_{r+1} = 0$. Thus the above sum can be written as
\begin{equation*}
\sum_{h=1}^r \sum_{v=1}^r  \sum_{a=m_{v+1}}^{m_v-1} {a+s-1 \choose s-1} 2^{\frac{2}{q} [m_h-a- 2(m_h-a)_+ ] } {2(m_h-a)_+ + s-1 \choose s-1}^{2/q}.
\end{equation*}
For $v \ge h$ we have $a \le m_h$. We can use \eqref{ineq_el} again to deduce
\begin{align*}
& \sum_{a=m_{v+1}}^{m_v-1} {a+s-1 \choose s-1} 2^{\frac{2}{q} [m_h-a- 2(m_h-a)_+ ] } {2(m_h-a)_+ + s-1 \choose s-1}^{2/q} \\ \ll_s & m_v^{s-1} \sum_{a=m_{v+1}}^{m_v-1} 2^{\frac{2}{q} [- (m_h-a) ] } {2(m_h-a) + s-1 \choose s-1}^{2/q} \\ \le & m_v^{s-1} \sum_{c=0}^\infty 2^{-2c/q} {2c + s-1 \choose s-1}^{2/q} \\ \ll_{s,q} & m_v^{s-1}.
\end{align*}

For $v < h$  we have $a > m_h$. We can use \eqref{ineq_el} again to deduce
\begin{align*}
& \sum_{a=m_{v+1}}^{m_v-1} {a+s-1 \choose s-1} 2^{\frac{2}{q} [m_h-a- 2(m_h-a)_+ ] } {2(m_h-a)_+ + s-1 \choose s-1}^{2/q} \\ \le & \sum_{a=m_{v+1}}^{m_v-1} {a+s-1 \choose s-1} 2^{\frac{2}{q} [m_h-a] } \\ \le & \sum_{a=m_{v+1}}^\infty 2^{\frac{2}{q} [m_h-a]} {a+s-1 \choose s-1} \\ \ll_s & 2^{\frac{2}{q} [m_h-m_{v+1}]} {m_{v+1} + s-1 \choose s-1}.
\end{align*}
Thus we have
\begin{align*}
& \sum_{a=0}^{m_1-1} {a+s-1 \choose s-1} \left( \sum_{h=1}^r 2^{m_h-a- 2(m_h-a)_+} {2(m_h-a)_+ + s-1 \choose s-1}  \right)^{2/q} \\ \ll_{s,q}  & \sum_{h=1}^r \left(\sum_{v=h}^r m_v^{s-1} + \sum_{v=1}^{h-1} 2^{\frac{2}{q} [m_h-m_{v+1}]} {m_{v+1} + s-1 \choose s-1} \right) \\ \le & r \sum_{v=1}^r m_v^{s-1} + \sum_{h=1}^r \sum_{c=0}^\infty 2^{-\frac{2}{q} c} {c+m_h+s-1 \choose s-1} \\ \ll_s & r \sum_{v=1}^r m_v^{s-1}.
\end{align*}
Substituting this result into \eqref{ineq_Lq_bound} we obtain
\begin{equation*}
\mathcal{L}_q^2(\cP_{N,s}) \ll_{s,q} \frac{r^{3-2/q}}{N^2} \sum_{v=1}^r m_v^{s-1},
\end{equation*}
which implies the result.
\end{proof}

\begin{remark}
Assume that $q$ is an even integer with $2 \le q < \infty$. A slight improvement of \eqref{Lq_seq} can be obtained using the interpolation of norms
\begin{equation*}
\mathcal{L}_z(\cP_{N,s}) \le (\mathcal{L}_2(\cP_{N,s}) )^\theta (\mathcal{L}_q(\cP_{N,s}) )^{1-\theta},
\end{equation*}
where $0 \le \theta \le 1$ and $$\frac{1}{z} = \frac{\theta}{2} + \frac{1-\theta}{q}.$$ This elementary inequality can be shown by applying H\"older's inequality to $\int |\delta(\cP_{N,s}, \cdot)|^{z \theta} |\delta(\cP_{N,s}, \cdot)|^{z (1-\theta)}$. Since the results of this paper also apply to the sequences constructed in \cite{DP12}, we obtain that one can explicitly construct infinite-dimensional sequences for which the projection onto the first $s$ coordinates of the first $N$ points satisfies
\begin{equation*}
\mathcal{L}_z(\cP_{N,s}) \ll_{z,s} \frac{(\log N)^{s/2 + 3/2-1/z}}{N}
\end{equation*}
for all $N \ge 2$ and $2 \le z < \infty$. Note that $2 \le z < \infty$ can be any real number.
\end{remark}

\section{Acknowledgements}

The author acknowledges the valuable discussions with F. Pillichshammer and H. Niederreiter which led to the results in this paper. The comments and suggestions by the anonymous referee improving the exposition of the paper are greatly appreciated. The author is supported by an Australian Research Council Queen Elizabeth 2 Fellowship.

\vspace{2cm}
\noindent{\bf Author's Address:}\\

\noindent Josef Dick, School of Mathematics and Statistics, The University of New South Wales, Sydney 2052, Australia.  Email: josef.dick@unsw.edu.au \\

\end{document}